\documentclass[11pt]{amsart}%
\usepackage{amsmath}%
\usepackage{amssymb}%
\usepackage{amscd}%
\usepackage{color}%
\usepackage[all]{xy}

\renewcommand{\Bbb}{\mathbb}  

\newcommand{\Q}{{\Bbb{Q}}}  
\newcommand{\R}{{\Bbb{R}}}  
\newcommand{\C}{{\Bbb{C}}}  
\newcommand{\Z}{{\Bbb{Z}}}  

\newcommand{\Oh}{\mathcal{O}}
\newcommand{\op}[1]{\operatorname{#1}}

\newcommand{\gl}{\mathfrak{gl}}

\newcommand{\U}{\mathrm{U}}
\newcommand{\Un}{\U(n)}

\newcommand{\sing}{\mathrm{sing}}

\theoremstyle{plain}
    \newtheorem{theorem}{Theorem}[section]

    \newtheorem{prop}[theorem]{Proposition}
    \newtheorem{lemma}[theorem]{Lemma}
    
     \newtheorem{cor}[theorem]{Corollary}
\theoremstyle{definition}
    
    \newtheorem{defn}[theorem]{Definition}

    \newtheorem{rem}[theorem]{Remark}
 
\def\Alphabet{A,B,C,D,E,F,G,H,I,J,K,L,M,N,O,P,Q,R,S,T,U,V,W,X,Y,Z}
\def\alphabet{a,b,c,d,e,f,g,h,i,j,k,l,m,n,o,p,q,r,s,t,u,v,w,x,y,z}
\def\endpiece{xxx}
\def\makeAlphabet[#1]{\expandafter\makeA#1,xxx,}
\def\makealphabet[#1]{\expandafter\makea#1,xxx,}
\def\makeA#1,{\def\temp{#1}\ifx\temp\endpiece\else%
\mkbb{#1}\mkfrak{#1}\mkbf{#1}\mkcal{#1}\expandafter\makeA\fi}%
\def\makea#1,{\def\temp{#1}\ifx\temp\endpiece\else\mkfrak{#1}\mkbf{#1}\expandafter\makea\fi}%
\def\mkbb#1{\expandafter\def\csname bb#1\endcsname{\mathbb{#1}}}
\def\mkfrak#1{\expandafter\def\csname fr#1\endcsname{\mathfrak{#1}}}
\def\mkbf#1{\expandafter\def\csname b#1\endcsname{\mathbf{#1}}}
\def\mkcal#1{\expandafter\def\csname c#1\endcsname{\mathcal{#1}}}
\def\makeop[#1]{\xmakeop#1,xxx,}
\def\mkop#1{\expandafter\def\csname #1\endcsname{{\mathrm{#1}}}} %
\def\xmakeop#1,{\def\temp{#1}\ifx\temp\endpiece\else\mkop{#1}\expandafter\xmakeop\fi}%
\makeAlphabet[\Alphabet]
\makealphabet[\alphabet]
\makeop[Alt,End,Ext,Hom,Sym,Tor]
\makeop[CH,Pic,div]
\makeop[Ker,Im,Coker,Coim,id,pr,isom]
\makeop[Spec,Spf,Spm]
\makeop[Re,Im]%
\makeop[dR,Nis,crys,rig,syn,tor,Cont,mot]
\makeop[Gal,Res,ab]
\makeop[pol]
\makeop[Var,an,Isoc,univ]%
\makeop[Gl,Sl,sl, U]
\makeop[Lie]
\makeop[Laz]
\makeop[tay,tr,Tr,sgn]
\makeop[Sat]

\def\isom{\cong}

\begin{document}
\title[Comparing volume forms]{Comparing natural volume forms on $\Gl_n$}   
\author{Annette Huber and Wolfgang Soergel}
\date{\today} 
\maketitle

\begin{abstract}
There are two natural choices for a volume form on the algebraic  group 
$\Gl_n/\Q$: the first is the integral form (unique up to sign), the other
is the product of the primitive classes in algebraic de Rham cohomology. 
We work out the explicit comparision factor between the two.
\end{abstract}

\section{Introduction}
Consider the group $\Gl_n$ and its Lie algebra $\gl_n$. Both are defined
over $\Z$. The isomorphism
\[ H^i(\gl_n,\Q)\to H^i_\dR(\Gl_n,\Q)\]
(\cite{Ho} Lemma 4.1)
can be used to define an integral structure in algebraic de Rham cohomology
as the image of integral Lie algebra cohomology.
\begin{defn}Let
\[ \rho_\Z^\dR\in H^{n^2}_\dR(\Gl_n,\Q)\]
be the image of a generator of
\[ H^{n^2}(\gl_n,\Z)=\bigwedge^{n^2}\gl_n^*\]
where $\gl_n^*$ is the $\Z$-dual of the integral Lie algebra $\gl_n$.
\end{defn}
Note that $\rho_\Z^\dR$ is only well-defined up to sign. 

Let $p_i^\dR\in H^{2i-1}_\dR(\Gl_n,\Q)$ be the primitive element
normalized as suspension of the universal Chern class $c_i^\dR\in H^{2i}(B\Gl_n,\Q)$. 
\begin{defn}We call
\[ \omega^\dR=p_1^\dR\wedge\dots\wedge p_n^\dR\in H^{n^2}_\dR(\Gl_n,\Q)\]
the {\em Borel element}.
\end{defn}
The Borel element occurs in his definition of a regulator on
higher algebraic $K$-theory of number fields. In \cite{borel} 
he relates it to special values of Dedekind $\zeta$-functions
of number fields, at least up to a rational factor.

The purpose of this note is to verify the following comparison result:
\begin{prop}\label{mainresult}
\[ \omega^\dR=\pm \left(\prod_{j=1}^{n}(j-1)!\right)\rho_\Z^\dR\]
\end{prop}
Our strategy is to use the comparison isomorphism between
de Rham cohomology and singular cohomology, which is compatible with 
Leray spectral sequences, products and Chern classes. The structure of singular
cohomology of $\Gl_n(\C)$ with integral coefficients is well-known and
in particular the product of the primitive classes is an integral
generator of $H^{n^2}_\sing(\Gl_n(\C),\Z)$. It remains to compare 
it with $\rho_\Z^\dR$. This is done by integrating the differential
form $\rho_\Z^\dR$ over a fundamental cycle, i.e., over $\Un$.

The interest for this result comes from an ongoing joint project of the
first author and G. Kings relating the unkown rational factor in Borel's
work to the Bloch-Kato conjecture for Dedekind-$\zeta$-functions.

\noindent {\em Acknowledgements:} We would like to thank A. Glang, S. Goette, G. Kings and M. Wendt for discussions and H. Klawitter for a numerical check in low degrees.

\section{Singular cohomology}
\begin{defn}Let $E\Gl_n$ be the simplicial scheme with $E_n\Gl_n=\Gl_n^k$
with boundary maps given by projections and degeneracies by diagonals. 
It carries a natural diagonal operation of $\Gl_n$.  
The classifying space $B\Gl_n$ is the quotient of $E\Gl_n$ by this action.
\end{defn}
We view $\Gl_n(\C)$ etc. as topological spaces with the analytic topology.
Let $\Un$ be the unitary group as real Lie group.
\begin{prop}[Borel]\label{sing}Let $c_j^\sing\in H^{2j}_\sing(B\Gl_n(\C),\Q)$ be the universal $j$-th Chern
class in singular cohomology. Let 
\[ s_j:H^{2j}_\sing(B\Gl_n(\C),\Z)\to H^{2j-1}_\sing(\Gl_n(\C),\Z)\]
be the suspension map. Let $p_j^\sing=s_j(c_j^\sing)$.
Then:
\begin{enumerate}
\item
\[ H^*_\sing(B\Gl_n(\C),\Z)=\Z[c_1^\sing,c_2^\sing,\dots,c_n^\sing]\]
as graded algebras.
\item
With $P_n=\bigoplus_{j=1}^n\Z p_j^\sing$ we have
\[ H^*_\sing(\Gl_n(\C),\Z)=\bigwedge^*_\Z P_n\]
as graded Hopf-algebras.
\end{enumerate}
\end{prop}
\begin{proof} 
Let $S^i$ be the $i$-sphere. 
Integral cohomology of the group is computed as
\[ H^*(S^{2n+1}\times S^{2n-1}\times\dots\times S^1,\Z)\]
in
\cite{Bo1} Proposition~9.1. This means it is an exterior 
algebra on generators $y_1,\dots,y_n$.
By loc. cit. Proposition~19.1 (b) $H^*_\sing(B\Gl_n(\C),\Z)$ is
a polynomial algebra on the same generators.

Let $T$ be the diagonal torus of $\Gl_n(\C)$ and
$W$ the Weyl group (i.e. the symmetric group).
Then we have (\cite{Hu} Ch. 18, Theorem 3.2)
\[ 
H^*_\sing(B\Gl_n(\C),\Z)= H^*_\sing(B T(\C),\Z)^W=\Z[c_1^\sing,\dots,c_n^\sing]
\]
This implies that $y_i$ can be identified with the
universal Chern class $c_i^\sing$.
\end{proof}

\begin{rem}This is the statement in the form usually used in algebraic
topology. From the point of view of complex or algebraic geometry it would be
more natural to view $c_j$ as an element of $H^{2j}_\sing(B\Gl_n(\C),(2\pi i)^j\Z)$. There is a hidden choice of $i$ or orientation on $\C$ behind the
translation from one point of view to the other.
\end{rem}

\begin{cor}\label{fund} The product
\[ \omega_\sing=p_1^\sing\wedge\dots p_n^\sing\]
is a generator of $H^{n^2}_\sing(\Gl_n(\C),\Z)$. It is the dual of
the fundamental class of 
\[ [\Un]\in H_{n^2}^\sing(\Un,\Z)\isom H_{n^2}^\sing(\Gl_n(\C),\Z)\]
\end{cor}
\begin{proof}The first statement is contained in the proposition. 
$\Un\subset \Gl_n(\C)$ is a homotopy equivalence. $\Un$ is compact, orientable and connected, hence $H_{n^2}^\sing(\Un,\Z)$
is generated by the manifold $\Un$ itself
(\cite{GH} Theorem 22.24).
\end{proof}

\section{De Rham cohomology}
\begin{defn}
Let $X$ be a smooth algebraic variety over $\Q$. Its {\em algebraic de
Rham cohomology} is defined as
\[ H^j_\dR(X)=H^j(X,\Omega^*_X)\]
the hypercohomology of the algebraic de Rham complex.
\end{defn}

Recall that there is a natural isomorphism of functors
\[ \sigma: H^j_\sing(X(\C),\Z)\otimes_\Z\C\to H^j_\dR(X)\otimes_\Q\C\]
It is induced by the inclusion  $\Z_X\to\C_X$  of sheaves for the analytic topology on $X(\C)$ and the quasi-isomorphism
\[ \C_X\to \Omega^{\an,*}_X\]
with the holomorphic de Rham complex (holomorphic Poincar\'e Lemma) on the 
one hand and the comparison between algebraic and holomophic de Rham cohomology on the other hand. In particular, $\sigma$ is compatible with products.

\begin{prop}Let $c_j^\dR\in H^{2j}_\dR(B\Gl_n)$ be the universal $j$-th Chern
class in algebraic de Rham cohomology. Let 
\[ s_j:H^{2j}_\dR(B\Gl_n)\to H^{2j-1}_\dR(\Gl_n)\]
be the suspension map. Let $p_j^\dR=s_j(c_j^\dR)$.
Then:
\begin{enumerate}
\item
\[ H^*_\dR(B\Gl_n)=\Q[c_1^\dR,c_2^\dR,\dots,c_n^\dR]\]
as graded algebras.
\item With $P_n=\bigoplus_{j=1}^n\Q p_j^\dR$ we have
\[ H^*_\dR(\Gl_n)=\bigwedge\Q^*P_n\]
as graded Hopf-algebras.
\end{enumerate}
\end{prop}
\begin{proof}There are different arguments for this fact. Once algebraicity of the Chern classes is known, the result follows directly from Proposition \ref{sing} and the existence of the comparison isomorphism. 
\end{proof}

\begin{prop}\label{comparison}
The comparison isomorphism $\sigma$ is compatible with Chern classes. More
precisely,
\[ \sigma((2\pi i)^jc_j^\sing)=c_j^\dR\]
\end{prop}
\begin{proof} Recall that the $j$-th Chern class  is the 
$j$-th elementary symmetric polynomial in the $1$-st Chern class of
diagonal torus (splitting principle). Hence it suffices to consider
the case $j=1$. 

For singular cohomology (or rather cohomology of sheaves on $\Gl_1(\C)=\C^*$)
consider the exact sequence of sheaves
\[ 0\to \Z\xrightarrow{2\pi i} \Oh^\an\xrightarrow{\exp} \Oh^{\an *}\to 1\]
$c_1^\sing$ is the image of the invertible function $z$ (the coordinate function of $\C^*$) under the connecting homomorphism.
For algebraic or holomorphic de Rham cohomology consider the morphism of complexes
\[ \Oh^*[-1]\to \Omega^*\hspace{3ex}f\mapsto \frac{df}{f}\]
$c_1^\dR=\frac{dz}{z}$ is the image of the invertible function $z$ under this
morphism of complexes. The two constructions are nearly (but not quite) compatible with
the definition of the comparison functor $\sigma$ which asks for $\Z$ to
be naturally embedded into the constant functions $\C\subset\Oh^\an$. 
This gives the factor $2\pi i$ as claimed.
\end{proof}

\begin{cor}\label{comparisontop}
Let as before $\omega^\dR=p_1^\dR\wedge\dots\wedge p_n^\dR\in H^{n^2}_\dR(\Gl_n)$. Then
\[ \sigma ((2\pi i)^{\frac{n(n+1)}{2}}\omega^\sing)=\omega^\dR\]
\end{cor}
\begin{proof}By Proposition \ref{comparison} and compatibility of
$\sigma$ with the suspension map we have
\[ \sigma ((2\pi i)^jp_j^\sing)=p_j^\dR\]
Moreover, $\sigma$ is compatible with products.
\end{proof}

\begin{lemma}\label{rho}
For $i,j=1,\dots,n$ let $z_{ij}$ be the natural coordinate on $n\times n$-matrices.
Recall that $\rho_\Z^\dR$ is the integral generator of 
$H^{n^2}(\gl_n,\Z)\subset H^{n^2}_\dR(\Gl_n)$. Then
\[ \rho_\Z=\frac{1}{\det^n}\bigwedge_{i,j=1}^n dz_{ij}\]
\end{lemma}
\begin{proof} 
The complex $\Omega^*(\Gl_{n,\Q})$ is quasi-isomorphic to 
$\bigwedge^*\gl_{n,\Q}^*$ where elements in $\gl_{n,\Q}^*$ are viewed
as left-invariant differential forms (see \cite{Ho} Lemma 4.1).
The differential form in the statement is clearly $\Gl_n$-invariant. In 
order to check that it is an integral basis, it suffices to restrict to the
tangent space of $1\in\Gl_n$. There it is the standard generator.
\end{proof}

\begin{rem}All computations are up to sign, hence we do not have to specify
a prefered ordering of the coordinates.
\end{rem}

\section{A volume computation}
\begin{prop}\label{volumecomp}
Let $z_{ij}$ be the standard holomorphic coordinates on $\Gl_n(\C)$. Then
\[ \int_{\Un}\frac{1}{\det^n}\bigwedge^n_{ij=1} dz_{ij}
=\pm \prod_{\nu=0}^{n-1}\frac{(2\pi\op {i})^{\nu+1}}{\nu!}
\]
\end{prop}

Before going into the proof, we review integration of differential forms over 
fibres of a bundle, thereby fixing notation.
Consider $p:X\to Y$ a fibre bundle with smooth compact fibres
of dimension $c$ and
a $C_\infty$-volume form $\omega$ on $Y$. Recall the definition
of the volume form $\int_p\omega$ on $Y$: for every
$y\in Y$ and tangent vectors $v_1,\dots,v_q\in T_yY$, the
volume form $\omega[v_1,\dots,v_q]$ on $p^{-1}(y)$ assigns to
all
 $x\in p^{-1}y$ and
$w_1,\dots,w_c\in  T_x p^{-1}(y)$ the value
\[ \omega[v_1,\dots,v_q](w_1,\dots,w_c)=\omega(w_1,\dots,w_c,\tilde{v}_1,\dots,\tilde{v}_q)\]
where $\tilde{v}_i$ is a preimage of $v_i$. The form $\omega[v_1,\dots,v_q]$
is independent of the choice of these $\tilde{v}_i$. Then
\[ \left(\int_p\omega\right)(v_1,\dots,v_q)=\int_{p^{-1}(y)}\omega[v_1,\dots,v_q]\]
\begin{proof}[Proof of Proposition \ref{volumecomp}:]
Recall 
\[ \rho_\Z=\frac{1}{\det^n}\bigwedge^n_{ij=1} dz_{ij}\]
We argue by induction on $n$. For $n=1$ we have 
\[ \int_{S^1}\frac{dz}{z}=2\pi i\]
by Cauchy's formula.

Suppose now the formula holds true for $n$. We abbreviate the
value by $\pm C(n)$. The claim reads
\[ C(n+1)=\pm \frac{(2\pi i)^{n+1}}{n!}C(n)\]

We consider the diagram
 \begin{displaymath}
    \xymatrix{
      A \ar@{|->}[r] & \op{diag}(1,A)&  (z_0, \ldots, z_n)\ar@{=}[d]\\
      {\op{U}} (n) \ar@{^{(}->}[r] & {\op{U}} (n+1) \ar[r]\ar[d]^-p & \mathbb
C^{n +1}\owns (x_0 + {i}y_0, \ldots, x_n +
      {i}y_n) \\
      &S^{2 n+1} \ar@{^{(}->}[r] &\mathbb R^{2n+2}\ar[u]^-\wr \owns (x_0,y_0,=
\ldots, x_n,y_n)\ar@{|->}[u]
   }
\end{displaymath}
with the left vertical $p$ given by application to the first vector of the
standard basis $ \vec{a}_0 =(1,0, \ldots, 0)^\top $ $\in \mathbb R^{2n+2}$. 

We integrate $\rho_\Z$ over the fibres.
The resulting form $\int_p\rho_\Z$ is $\U(n+1)$-invariant and uniquely
determined by its value in $\vec{a}_0$, which we are going to compute.
Let $\vec{v}_1,\dots,\vec{v}_{2n+1}\in  T_{\vec{a}_0}S^{2n+1}$ be tangent vectors.
Then $\rho_\Z[v_1,\dots,v_{2n+1}]$ is an $\Un$-invariant form and uniquely
determined by its value in the unit matrix $E$.

We choose as basis of the tangent space of
$S^{2n+1}$ in $\vec a_0$ the other vectors in the standard basis of $\R^{2n+2}$
and denote them 
\[ 
\vec b_0, \vec a_1, \vec b_1, \ldots, \vec a_n,
  \vec b_n .
 \]
Let $\rho_0$ be the unique $\U(n+1)$-equivariant form on $S^{2n+1}$ with
\[ \rho_0(\vec{b}_0,\vec{a}_1,\dots,\vec{b}_n)=1\]

For later use, we record that the surface of the unit ball in dimension $2n+2$ is computed
by 
\begin{equation}\tag{*}\label{vol} \int_{S^{2n+1}}\rho_0=2\frac{\pi^{n+1}}{n!}\end{equation}

We have to choose preimages in the tangent space ${ T}_E {\op{U}}(n+1)$ 
for these vectors. We can use arbitrary hermitian matrices 
$A_\nu$ for $1 \leq \nu\leq n$ und $B_\nu$ for $0 \leq \nu \leq n$ such that
\[ A_\nu \vec a_0 =\vec a_n\hspace{3ex}B_\nu\vec a_0 = \vec b_\nu.\]
A simple choice are the complex matrices
\begin{align*}
A_\nu &= E_{\nu 0} - E_{0\nu}\hspace{5ex}\nu \geq 1\\
B_\nu &={i}E_{\nu 0} + {i} E_{0\nu}\hspace{3ex}\nu \geq 1\\
B_0 &={i} E_{00}
\end{align*}
Here we are using the usual notation for the standard basis of the matrix ring over $\C$ but with indices starting from $0$. 

It is more convenient to
pass to complexified tangent spaces. ${ T}^{\mathbb C}_{\vec a_{0}} S^{2n +1}$ has the simpler basis 
\begin{align*}
&\vec{v}_0=i\vec{b_0}\\
&\vec{v}_\nu=( \vec a_{\nu} - {i}\vec b_\nu)/2\hspace{7ex}\nu=1,\dots,n\\ 
&\vec{v}_{n+\nu}=( \vec a_\nu + {i} \vec b_\nu)/2\hspace{4ex}\nu=1,\dots,n
\end{align*}
Its lift
to ${ T}^{\mathbb C}_{E} {\op{U}}(n+1)$ is given by 
\begin{align*}
&\tilde{v}_0=-E_{00}\\
&\tilde{v}_\nu=+E_{\nu 0}\hspace{7ex}\nu=1,\dots, n\\
&\tilde{v}_{n+\nu}=-E_{0\nu}\hspace{4ex}\nu=1,\dots,n
\end{align*}
By evaluating in a standard basis of $ T_E^\C\op{U}(n)$ we get
\[ \rho_\Z[\vec{v}_0,\dots,\vec{v}_{2n}]=\pm dz_{11} \wedge \ldots \wedge dz_{nn}
=\rho_\Z^{\Un}\]
By inductive hypothesis this implies
\[ \left(\int_p\rho_\Z\right)(\vec{v}_0,\dots,\vec{v}_{2n})=\pm C(n)\]
We now translate back to the original basis.
We easily find for $\nu\geq 1$
\[ \vec{v}_\nu\wedge \vec{v}_{n+\nu}=\frac{1}{4}(\vec{a}_\nu-i\vec{b}_\nu)\wedge (\vec{a}_\nu+i\vec{b}_\nu)=\frac{i}{2}\vec{a}_\nu\wedge\vec{b}_\nu\] 
and hence
\begin{gather*}
\bigwedge_{i=0}^{2n}\vec{v}_i= \pm( \vec b_0 \wedge \vec a_1\wedge \vec b_1 \wedge \ldots
    \wedge \vec a_n \wedge \vec b_\nu ) \cdot 
    \frac{i^{n+1}}{2^n}\Rightarrow\\
\left(\int_p\rho_\Z\right)(\vec{b}_0,\vec{a}_1,\vec{b}_1,\dots,\vec{b}_{n})=\pm C(n) i^{n+1}2^n\Rightarrow\\
\int_p\rho_\Z=\pm C(n)i^{n+1}2^n\rho_0
\end{gather*}
Together with equation (\ref{vol}) for the unit sphere this yields
\begin{align*}
C(n+1)&= \int_{\U(n+1)}\rho_\Z=\int_{S^{2n+1}}\left(\int_p\rho_\Z\right)\\
&= \pm C(n)i^{n+1}2^{n}\int_{S^{2n+1}}\rho_0\\
&=\pm C(n)\frac{(2\pi i)^{n+1}}{n!}   
\end{align*}
This proves the claim.
\end{proof}

\section{Proof of the main result}
\begin{proof}[Proof of Proposition \ref{mainresult}]
We want to compare the elements $\rho_\Z^\dR$ (see Lemma \ref{rho})
and $\omega^\dR$ (see Corollary \ref{comparisontop}) in $H^{n^2}_\dR(\Gl_n)$.
Let $\alpha\in\Q^*$ such that
\[ \omega_\dR=\alpha \rho_\Z^\dR\]
By Corollary \ref{comparisontop}
\[ \sigma (\omega^\sing)=(2\pi i)^{-\frac{n(n+1)}{2}}\alpha\rho_\Z^\dR\]
The comparison isomorphism between singular cohomology and holomorphic de Rham cohomology
can be reformulated as integration.
By  Corollary \ref{fund}, Lemma \ref{rho} and Proposition \ref{volumecomp}, this means
\[ (2\pi i)^{\frac{n(n+1)}{2}}\alpha^{-1}=\int_{\Un}\frac{1}{\det^n}\wedge_{i,j=1}^n dz_{ij}
=\pm \prod_{\nu=0}^{n-1}\frac{(2\pi i)^{\nu +1}}{\nu!}
\]
where $z_{ij}$ are the holomorphic coordinates on the space on $n\times n$ matrices.
Hence 
\[ \alpha=\pm \prod_{\nu=0}^{n-1}\frac{1}{\nu!}\]
as claimed.
\end{proof}

\end{document}